\documentclass[12pt]{amsart}

\usepackage{graphicx,eepic}
\usepackage{a4wide}
\usepackage{enumerate}
\usepackage{amsmath,amsfonts,amssymb}

\numberwithin{equation}{section} \numberwithin{table}{section}

\newtheorem{lemma}{Lemma}[section]

\newtheorem{prop}[lemma]{Proposition}
\newtheorem{thm}[lemma]{Theorem}
\newtheorem{cor}[lemma]{Corollary}
\theoremstyle{definition}
\newtheorem{defn}[lemma]{Definition}
\newtheorem{example}[lemma]{Example}

\theoremstyle{remark}
\newtheorem{rmk}[lemma]{Remark}

\newcommand{\de}{\delta}

\newcommand{\al}{\alpha}
\newcommand{\ga}{\gamma}
\newcommand{\e}{\varepsilon}

\newcommand{\om}{{\omega'}}

\newcommand{\BN}{\mathbb N}

\renewcommand{\epsilon}{\varepsilon}

\begin{document}

\title[On cycles for the doubling map]{On cycles for the doubling map \\ which are disjoint from an interval}
\author{Kevin G. Hare}
\address{Department of Pure Mathematics \\
University of Waterloo \\
Waterloo, Ontario \\
Canada N2L 3G1}
\thanks{Research of K. G. Hare supported, in part by NSERC of Canada.}
\email{kghare@uwaterloo.ca}
\author{Nikita Sidorov}
\address{School of Mathematics \\
The University of Manchester \\
Oxford Road, Manchester\\
 M13 9PL, United Kingdom.}
\email{sidorov@manchester.ac.uk}

\date{\today}

\begin{abstract}
Let $T:[0,1]\to[0,1]$ be the doubling map and let $0<a<b<1$. We say that an integer $n\ge3$ is bad for $(a,b)$ if all $n$-cycles for $T$ intersect $(a,b)$. Let $B(a,b)$ denote the set of all $n$ which are bad for $(a,b)$. In this paper we completely describe the sets:
\[
D_2=\{(a,b) : B(a,b)\,\text{is finite}\}
\]
and
\[
D_3=\{(a,b) : B(a,b)=\varnothing\}.
\]
In particular, we show that if $b-a<\frac16$, then $(a,b)\in D_2$, and if $b-a\le\frac2{15}$, then $(a,b)\in D_3$, both constants being sharp.
\end{abstract}

\maketitle

\section{Introduction and summary}

Let $T:[0,1]\to [0,1]$ be the doubling map, i.e.,
\begin{equation*}
Tx=\begin{cases} 2x, & x\in[0,1/2],\\
2x-1, & x\in(1/2, 1].
\end{cases}
\end{equation*}
Assume $0<a<b<1$ and put
\[
\mathcal J(a,b)=\{x\in[0,1] : T^nx\notin (a,b)\ \text{for all}\ n\ge0\}.
\]
In other words, $\mathcal J(a,b)$ is the set of all $x$ whose $T$-orbits are disjoint from $(a,b)$. Thus, $T|_{\mathcal J(a,b)}$ is what is usually referred to as an ``open map'' or a ``map with a hole''. It is obvious that $\{0,1\}\subset\mathcal J(a,b)$.

It is intuitively clear that if $b-a$ is ``small'', then $\mathcal J(a,b)$ is ``large'' and vice versa. Such claims in their precise quantitative form have been obtained in the recent papers \cite{GS13, SSC}. Specifically, the following two sets have been described in \cite{GS13}:
\[
D_0=\{(a,b)\in (1/4,1/2)\times(1/2,3/4) : \mathcal J(a,b)\neq\{0,1\}\}
\]
and
\[
D_1=\{(a,b)\in (1/4,1/2)\times(1/2,3/4) : \mathcal J(a,b)\ \text{is uncountable}\}.
\]
See Figure~\ref{fig:D0123} on page~\pageref{fig:D0123} for $D_0$ and $D_1$. The reason why we can confine ourselves to $(a,b)\in (1/4,1/2)\times(1/2,3/4)$ without losing anything interesting is the following result:

\begin{lemma} {\rm (}\cite[Lemma~1.1]{GS13}{\rm )}
\begin{enumerate}
\item If $a<1/4, b>1/2$ or $a<1/2, b>3/4$, then $\mathcal J(a,b)=\{0,1\}$.
\item If $b<1/2$ or $a>1/2$, then $\dim_H \mathcal J(a,b)>0$.
\end{enumerate}
\end{lemma}

In the present paper we will be interested in cycles (i.e. finite orbits) for $T$ which are disjoint from an interval. Let us first introduce the following sets which are closely related to $D_0$ and $D_1$:
\[
D_0'=\{(a,b)\in (1/4,1/2)\times(1/2,3/4) : \exists\,\text{a non-trivial cycle disjoint from}\ (a,b)\}
\]
and
\[
D_1'=\{(a,b)\in (1/4,1/2)\times(1/2,3/4) : \exists\,\text{infinitely many cycles disjoint from}\ (a,b)\}
\]
(from here on by a ``cycle'' we will mean a prime $T$-cycle). In the definition of $D_0'$ we do not include the trivial 1-cycles of $\{0\}$ or $\{1\}$.
Note first that it is obvious that if $(a,b)\notin D_0$, then $(a,b)\notin D_0'$, whence $D_0'\subset D_0$. On the other hand, if $(a,b)$ is an interior point of $D_0$, then a cycle disjoint from $(a,b)$ exists -- see \cite[Theorem~2.7]{GS13}. Furthermore, the set $\{(a,a+1/4) : a\in\mathcal S\}$ is a subset of $D_0$ but it does not contain any cycles, and these are the only points on the boundary of $D_0$ with this property. (For the definition of $\mathcal S$ see Section~\ref{sec:D2}.) Hence
\[
D_0'\subsetneq D_0 \subsetneq \mathrm{cl}(D_0').
\]
where $\mathrm{cl}$ is the closure of the set. In particular, the interiors of $D_0$ and $D_0'$ coincide. Similarly, in view of \cite[Theorem~2.16]{GS13},
\[
D_1'\subsetneq D_1 \subsetneq \mathrm{cl}(D_1').
\]
Thus, the sets $D_0$ (resp. $D_1$) are ``almost'' the ones where for $(a,b)$ there are at least one (resp. infinitely many) disjoint cycles.

One may see this model as follows: take some large interval $(a,b)$ and begin shrinking it from both ends. At some point in time one gets a disjoint cycle, then infinitely many of those, and then (apparently!) for any $n$ there will be an $n$-cycle disjoint from $(a,b)$. This is analogous to the famous ``period three implies chaos'' statement -- see, e.g., \cite{LiY}; in fact, it is more than just an analogy, it is a generalization.

More precisely, if we assume $b=1-a$ (so our shrinking intervals are always symmetric about $1/2$), then it follows from the main result of \cite{ACS} that with the increase of $a$ towards $1/2$, the lengths of cycles disjoint from $(a,1-a)$ appear in exactly the classical Sharkovski\u{\i} order, with period~3 being indeed last to appear at $a=3/7$. The case we consider in the present paper is, generally speaking, asymmetric, so our first goal will be to determine what curve in the plane $(a,b)$ is a natural analogue of $3/7$ -- see Remark~\ref{rmk:D3} and Proposition~\ref{prop:exit} in Section~\ref{sec:D3}.

Note first that $n=2$ needs to be excluded, since there is only one 2-cycle, namely, $\{1/3, 2/3\}$, so one can take $(a,b)=(1/3-\e, 1/3+\e)$, and the 2-cycle is never disjoint from $(a,b)$, which is not particularly interesting.

To simplify our definitions, we say that an integer $n\ge3$ is {\em bad for $(a,b)$} if each $n$-cycle for $T$ has a non-empty intersection with $(a,b)$. Let $B(a,b)$ denote the set of all $n\ge3$ which are bad for $(a,b)$.
Put
\[
D_3=\{(a,b)\in (0,1)\times(0,1) : B(a,b)=\varnothing\}.
\]
Thus, $(3/7,4/7)\in D_3$.
Unlike the case of $D_0$ and $D_1$, there is some structure to $D_3$
    outside of $(1/4, 1/2) \times (1/2, 3/4)$.
We will show in Section~\ref{sec:D3} that this structure is very easily
    explained and all interesting structure will still lie within
    $(1/4, 1/2) \times (1/2, 3/4)$.
So, although the set is defined on the larger range $(0,1)\times(0,1)$,
    we will quite often restrict out attention to the range
    $(1/4, 1/2) \times (1/2, 3/4)$.
See Figure~\ref{fig:D3} for $D_3$, both on $(0,1)\times(0,1)$ and $(1/4,1/2)\times(1/2,3/4)$.

We will show in Section~\ref{sec:D3} that the boundary of $D_3$ is made up of a finite number of horizontal and vertical lines -- see
Figure~\ref{fig:D3}.

\begin{figure}
\includegraphics[width=190pt,height=237pt]{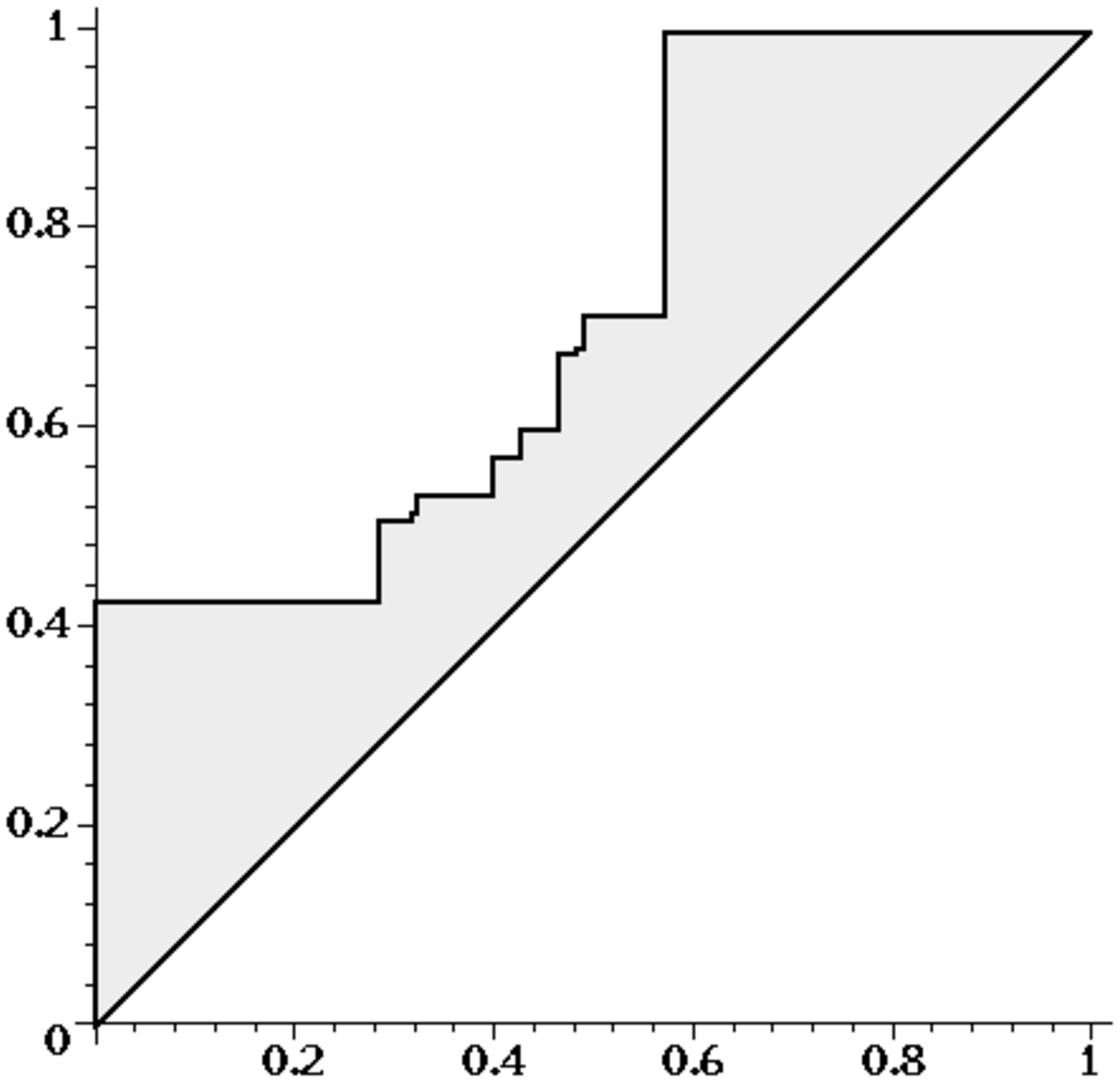} \ \ \ \  \ \  \  \
\includegraphics[width=190pt,height=237pt]{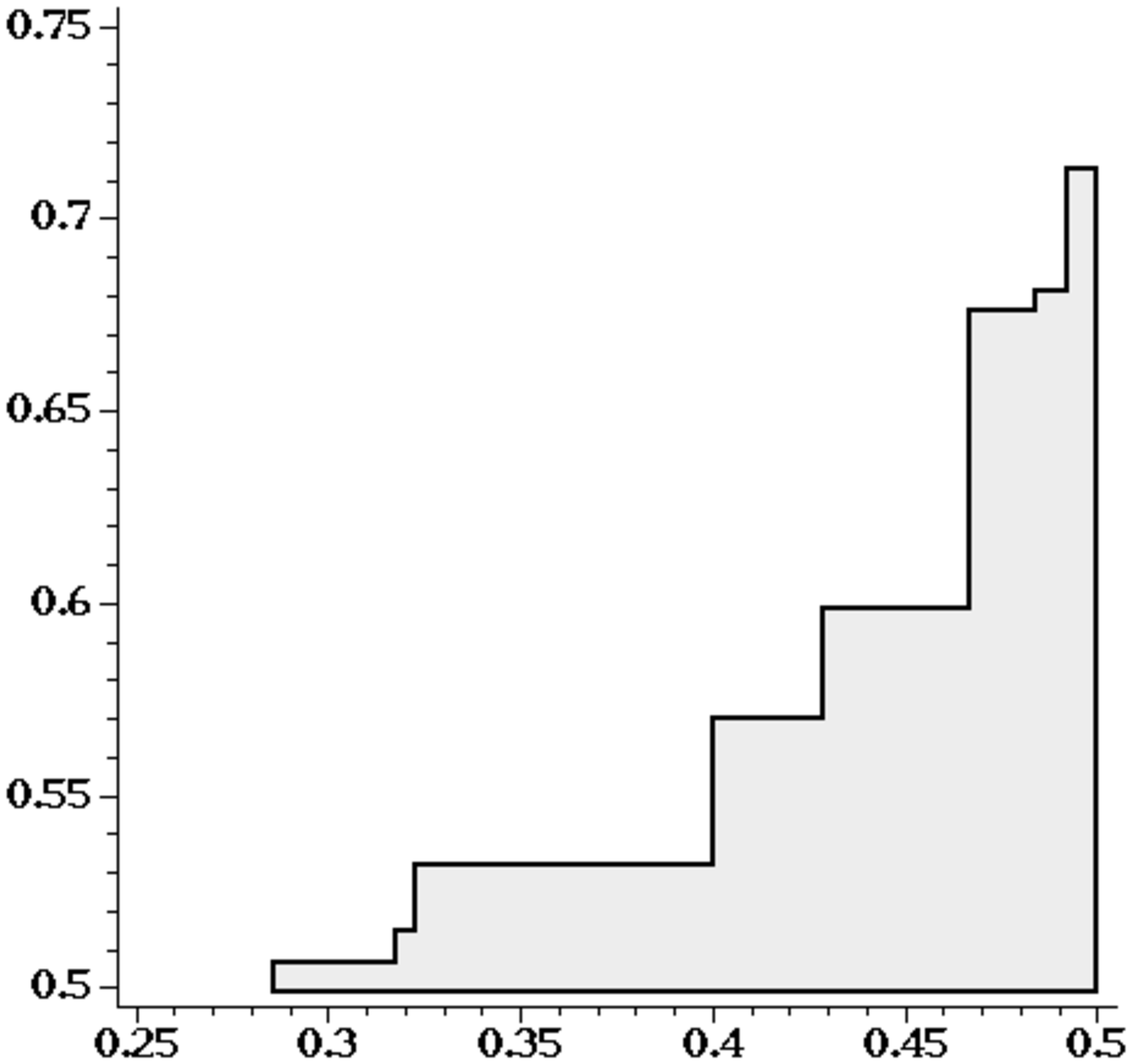}
\caption{The set $D_3$ in $(0,1)\times(0,1)$ and $(1/4,1/2)\times(1/2,3/4)$}
\label{fig:D3}
\end{figure}

There is another important milestone in the Sharkovski\u{\i} order, namely, the threshold below which all the even periods already exist but none of the odd ones does. In the symmetric model $b=1-a$ this milestone is $a=5/12$ whose binary expansion is $01(10)^\infty$ -- this follows immediately from \cite[Proposition~2.16]{ACS}, in which the critical values of $a$ for all the periods are computed. We introduce its natural analogue for the asymmetric case:
\[
D_2=\{(a,b)\in (1/4,1/2)\times(1/2,3/4) : B(a,b)\,\text{is finite}\}.
\]
Note that the restriction $(a,b)\in (1/4,1/2)\times(1/2,3/4)$ is, again, natural, since, as with $D_0$ and $D_1$, if $(a,b)$ contains $(1/4,1/2)$ or $(1/2, 3/4)$, there cannot be any disjoint cycles for $(a,b)$. Also, if $b<1/2$ or $a>1/2$, then $B(a,b)$ is always finite (or empty). Indeed, let $b<1/2$ (the case $a>1/2$ is completely analogous); here one can take $x=\frac{2^{n-1}-1}{2^n-1}$ and it is easy to check that it is a part of the $n$-cycle
    $\left\{\frac{2^n-2}{2^n-1}, \frac{2^n-3}{2^n-1}, \frac{2^n-5}{2^n-1}, \frac{2^n-9}{2^n-1}, \cdots, \frac{2^{n-1}-1}{2^n-1}\right\}$
which lies in $\bigl[\frac{2^{n-1}-1}{2^n-1}, 1\bigr]$, which is disjoint from $(a,b)$ for all sufficiently large $n$.

Although most of the boundary of $D_2$ is made up of horizontal and vertical lines, it is in fact made up of an infinite number of horizontal and vertical lines (one associated to each rational number), creating a kind of Devil's staircase. The precise structure of this set will be discussed in Section~\ref{sec:D2}. (See Figure \ref{fig:D2} as a shape of things to come.)

\begin{figure}
\includegraphics[width=250pt,height=313pt]{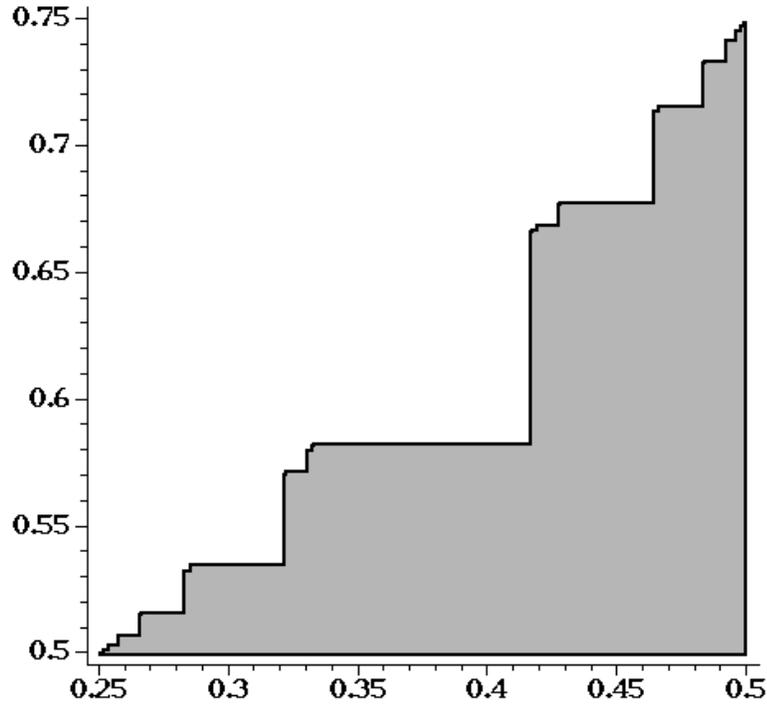}
\caption{The set $D_2$}
\label{fig:D2}
\end{figure}

Thus, one can say that while \cite{GS13} is about the initial part of the ``asymmetric Sharkovski\u{\i} order'' which generalizes the usual period doubling in three different ways (see \cite[Section~4.3]{GS13} for a detailed exposition), the present paper is about the ``final stretch'' of such orders, which generalizes the usual sequence of odd numbers in the reverse order.

Thus, the main reason why we believe a detailed study of the sets $D_2$ and $D_3$ is interesting is the fact that these sets are cornerstones of the generalized Sharkovski\u{\i} order, which appears to be an exciting object per se.

Regarding what happens ``in between'' -- generalizing the range in the Sharkovski\u{\i} order between getting all the powers of two and all the even numbers -- note that $\partial D_1$ and $\partial D_2$ have a substantial intersection (see Remark~\ref{rmk1} and Figure~\ref{fig:D0123} below). This means that for most patterns, shrinking $(a,b)$ results in simultaneously obtaining infinitely many disjoint cycles as well as finitely many bad $n$.

Note, in \cite{Alseda}, the authors studied for which $n$ does the Lorentz-like map have an $n$-cycle. Our case is somewhat different, because we are also interested in avoiding
    holes.

\section{The set $D_3$}
\label{sec:D3}

In this section we will show that the boundary of $D_3$ is composed of a finite number of horizontal and vertical lines.
We will give a precise description for the locations of these lines.

For any $(w_1,w_2,\dots)\in\{0,1\}^{\mathbb N}$ put
\[
x=\pi(w_1,w_2,\dots)=\sum_{j=1}^\infty w_j2^{-j},
\]
i.e., the dyadic (binary) expansion of $x$. From here on for the sake of notation we will not distinguish between the numbers in $[0,1]$ and their dyadic expansions.

Since we plan to work closely with 0-1 words, we need some definitions and basic results from combinatorics on words -- see \cite[Chapter~2]{Loth} for a detailed exposition. For any two finite words $u=u_1\dots u_k$ and $v=v_1\dots v_n$ we write $uv$ for their concatenation $u_1\dots u_k v_1\dots v_n$. In particular, $u^m=u\dots u$ ($m$ times) and $u^\infty=uuu\dots=\lim_{n\to\infty}u^n$, where the limit is understood in the topology of coordinate-wise convergence.
We will denote by $u^*$ the set of words
    $\{\lambda, u, u^2, u^3, u^4, \dots\}$, where $\lambda$ is the empty word.

From here on by a ``word'' we will mean a word whose letters are 0s and 1s. Let $w$ be a finite or infinite word.
We say that a finite or infinite word $u$ is {\em lexicographically smaller than} a word $v$ (notation: $u\prec v$) if either $u_1<v_1$ or there exists $n\ge1$ such that $u_i\equiv v_i$ for $i=1,\dots, n$ and $u_{n+1}<v_{n+1}$.
We notice that if $u \prec v$ then $\pi(u) \leq \pi(v)$ with equality only
    if $u = w 0 1^\infty$ and $v = w 1 0^\infty$ for some finite word $w$.

Recall $D_3 := \{(a,b) : B(a,b) = \emptyset\}$ where $B(a,b)$ is the set of
    bad $n$ for $(a,b)$.
We make two observations:
\begin{enumerate}
\item if $(a,b) \not\in D_3$ then $(a-\de, b+\e) \not \in D_3$
    for any non-negative $\e$ and $\de$; \label{ob:1}
\item if $(a,b)\in D_3$, then $(a+\de, b-\e)\in D_3$ for any
    non-negative $\e$ and $\de$.   \label{ob:2}
\end{enumerate}
We will show that $D_3$ has a very simple structure, namely that $D_3$'s boundary is composed of finitely many horizontal and vertical lines.

\begin{defn}\label{def:corner}
We will say that $(a,b)$ is a {\em corner of $D_3$} if
    for $(a', b')$ sufficiently close to $(a,b)$ we have
\begin{itemize}
\item if $a' > a$ or $b' < b$ then $(a', b') \in D_3$;
\item if $a' < a$ and $b' > b$ then $(a', b') \not\in D_3$.
\end{itemize}
\end{defn}

\begin{defn}\label{def:anti-corner}
We will say that $(a,b)$ is an {\em anti-corner of $D_3$} if
    for $(a', b')$ sufficiently close to $(a,b)$ we have
\begin{itemize}
\item If $a' > a$ and $b' < b$ then $(a', b') \in D_3$
\item If $a' < a$ or $b' > b$ then $(a', b') \not\in D_3$
\end{itemize}
\end{defn}

See Figure~\ref{fig:Corner} for a sketch. It is worth noting that we make no comment on $a' = a$ or $b' = b$.
Although $D_3$ is closed, $D_2$ is open, and we wish to reuse the
    definition of corner later on for $D_2$.
It is not clear a priori that $D_3$ will have corners and anti-corners as
    we have defined them.  For example $\{(x,y) : x > y\}$ does not have
    either.
We will show that in fact $D_3$ is made up of a finite number of corners
    and anti-corners.
Further we will show that these corners and anti-corners completely
    describes $D_3$.


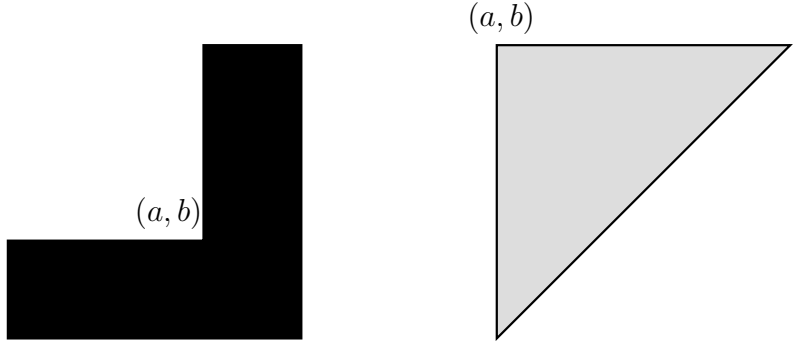
\begin{figure}[t]
    \centering \unitlength=1.3mm
    \texture{c0c0c0c0 0 0 0 0 0 0 0
        c0c0c0c0 0 0 0 0 0 0 0
        c0c0c0c0 0 0 0 0 0 0 0
        c0c0c0c0 0 0 0 0 0 0 0 }

    \centering \unitlength=1.3mm
    \begin{picture}(80,53)(0,-7)
            \thicklines

            \path(0,0)(0,10)(20,10)(20,30)(30,30)(30,0)(0,0)

            \shade\path(50,0)(50,30)(80,30)(50,0)

            \put(13,12){$(a,b)$}
            \put(47,32){$(a,b)$}


    \end{picture}
    \caption{A corner (left) and and an anti-corner (right)}
    \label{fig:Corner}
\end{figure}

\begin{thm}
\label{thm:D3 corner}
There are $9$ corners of $D_3$.
They are
\[
\begin{array}{rclrclrcl}
(a_1, b_1) & = &\left(\frac{2}{7}, \frac{3}{7}\right)  \hspace{10pt} &
(a_2, b_2) & = & \left(\frac{20}{63}, \frac{32}{63}\right) \hspace{10pt} &
(a_3, b_3) & = & \left(\frac{10}{31}, \frac{16}{31}\right) \\
(a_4, b_4) & = & \left(\frac{2}{5}, \frac{8}{15}\right) &
(a_5, b_5) & = & \left(\frac{3}{7}, \frac{4}{7}\right) &
(a_6, b_6) & = & \left(\frac{7}{15}, \frac{3}{5}\right) \\
(a_7, b_7) & = & \left(\frac{15}{31}, \frac{21}{31}\right) &
(a_8, b_8) & = & \left(\frac{31}{63}, \frac{43}{63}\right) &
(a_9, b_9) & = & \left(\frac{4}{7}, \frac{5}{7}\right) \\
\end{array}
\]
Letting $a_0 = 0$ and $b_{10} = 1$, and the $a_i$ and $b_i$ as above,
there are 10 anti-corners of $D_3$.  They are:
\[
(a_0, b_1), (a_1, b_2), (a_2, b_3), (a_3, b_4), (a_4, b_5),
(a_5, b_6), (a_6, b_7), (a_7, b_8), (a_8, b_9),  (a_9, b_{10}).
\]
\end{thm}

\begin{cor}
If $b-a\le\frac{2}{15}$, then $(a,b)\in D_3$. On the other hand, if $(a,b)\in D_3$, then $b-a\le \frac37$.
\end{cor}

\begin{proof}
This follows by noticing that $\min(b_i-a_i) = b_4-a_4 = \frac{2}{15}$
    and
$\max(b_{i+1}-a_i) = b_1-a_0 = \frac{3}{7}$.
\end{proof}

\begin{proof}[Proof of Theorem~\ref{thm:D3 corner}]
We will show
\begin{itemize}
\item $(a_i - \e, b_i+\e) \not \in D_3$ for all $\e>0$,
    by explicitly finding an $n$ such that all $n$-cycles
    intersect $(a_i - \e, b_i + \e)$.
\item $(a_i, b_{i+1}) \in D_3$ by explicitly giving for each $n \geq 3$ an
    $n$-cycle that avoids $(a_i, b_{i+1})$.
\end{itemize}
These two claims, combined with the observations~\ref{ob:1} and \ref{ob:2} on page~\pageref{ob:1}, prove Theorem~\ref{thm:D3 corner}.
To see this, assume that we have shown for some fixed $i$ the
\begin{enumerate}[(a)]
\item $(a_i-\e, b_i+\e) \not \in D_3$,
\label{pnt1}
\item $(a_i, b_{i+1}) \in D_3$.
\label{pnt2}
\item $(a_{i+1}-\e, b_{i+1}+\e) \not \in D_3$ and
\label{pnt3}
\end{enumerate}
We see from \eqref{pnt2} above, and observation \ref{ob:2} that for all
    $a' > a_i$ and $b' < b_{i+1}$ that $(a', b') \in D_3$.
We see from \eqref{pnt1} above, and observation \ref{ob:1} that for all
    $a' < a_i$ and $b' > b_i$ that $(a', b') \not \in D_3$.
As $b_i < b_{i+1}$, this shows that for $b'$ sufficiently close to $b_i$
    and $a' < a_i$ that $(a', b') \not \in D_3$.
Similarly, we see that for $a'$ sufficiently close to $a_i$ and $b' > b_{i+1}$
    we have that $(a', b') \not \in D_3$.
This shows that $(a_i, b_{i+1})$ is an anti-corner.

In a similar fashion, the two claims at the start of the proof would show that
    each $(a_i, b_i)$ is a corner.

We will show that $D_3$ has a very simple structure, namely that $D_3$'s boundary is composed of finitely many horizontal and vertical lines.
To see this we see that if we show all of the points $(a_i, b_i)$ are
    corners and $(a_i, b_{i+1})$ are anti-corners, we see that the line
    from $(a_i, b_i)$ to $(a_i, b_{i+1})$ is on the boundary of $D_3$.
Similarly the line from $(a_i, b_i)$ to $(a_{i-1}, b_i)$ is also on
    the boundary of $D_3$.
This in turn shows that there cannot be any other corners or anti-corners,
    which proves the result.

The first part is demonstrated in Table~\ref{tab:Corner}.
Here we give $(a_i, b_i)$, the $n$ for which all $n$-cycles intersect
    $(a_i-\e, b_i+\e)$.
For $i = 1, 2, \dots, 5$ we also give all of the $n$-cycles.
For each $n$-cycle we indicate  in bold
    the term in the orbit that intersects $(a_i-\e, b_i+\e)$.
Note that in some cases there are multiple terms.

Consider, for instance, the special case of showing $\left(\frac27-\e, \frac37+\e\right) \not\in D_3$.
We see that there are only two different $3$-cycles.
One of these $3$-cycles contains $\frac27$ and the other contains $\frac37$.
Hence this interval will always contain a $3$-cycle and hence is not
    in $D_3$.

{\footnotesize
\begin{table}
\begin{tabular}{lll}
Corner & $n$ & $n$-cycles \\ \hline
$(\frac{2}{7}, \frac{3}{7})$ &
    3 &
    $\{1/7, \mathbf{2/7}, 4/7\}$ \\ & &
    $\{\mathbf{3/7}, 5/7, 6/7\}$ \\ \\

$(\frac{20}{63}, \frac{32}{63})$ &
    6 &
    $\{5/63, 10/63, 17/63, \mathbf{20/63}, {34/63}, 40/63\}$ \\ & &
    $\{11/63, \mathbf{22/63}, \mathbf{25/63}, {37/63}, 44/63, 50/63\}$ \\ & &
    $\{1/21, 2/21, 4/21, \mathbf{8/21}, {11/21}, 16/21\}$ \\ & &
    $\{13/63, 19/63, \mathbf{26/63}, {38/63}, 41/63, 52/63\}$ \\ & &
    $\{1/9, 2/9, \mathbf{4/9}, {5/9}, 7/9, 8/9\}$ \\ & &
    $\{\mathbf{23/63}, \mathbf{29/63}, {43/63}, 46/63, 53/63, 58/63\}$ \\ & &
    $\{5/21, \mathbf{10/21}, {13/21}, 17/21, 19/21, 20/21\}$ \\ & &
    $\{\mathbf{31/63}, 47/63, 55/63, 59/63, 61/63, 62/63\}$  \\ &&
    $\{1/63, 2/63, 4/63, 8/63, 16/63, \mathbf{32/63}\}$ \\  \\

$(\frac{10}{31},  \frac{16}{31})$ &
    5 &
    $\{5/31, 9/31, \mathbf{10/31}, 18/31, 20/31\}$ \\ & &
    $\{3/31, 6/31, \mathbf{12/31}, 17/31, 24/31\}$ \\ & &
    $\{\mathbf{11/31}, \mathbf{13/31}, 21/31, 22/31, 26/31\}$ \\ & &
    $\{7/31, \mathbf{14/31}, 19/31, 25/31, 28/31\}$ \\ & &
    $\{\mathbf{15/31}, 23/31, 27/31, 29/31, 30/31\}$ \\ & &
    $\{1/31, 2/31, 4/31, 8/31, \mathbf{16/31}\}$ \\  \\

$(\frac{2}{5}, \frac{8}{15})$ &
    4 &
    $\{1/5, \mathbf{2/5}, 3/5, 4/5\}$ \\ & &
    $\{\mathbf{7/15}, 11/15, 13/15, 14/15\}$ \\ &&
    $\{1/15, 2/15, 4/15, \mathbf{8/15}\}$ \\ \\

$(\frac{3}{7}, \frac{4}{7})$ &
    3 &
    $\{\mathbf{3/7}, 5/7, 6/7\}$ \\  &&
    $\{1/7, 2/7, \mathbf{4/7}\}$ \\  \\
$(\frac{7}{15}, \frac{3}{5})$ &
    4& as above \\ \\

$(\frac{15}{31}, \frac{21}{31})$ &
    5 & as above\\ \\

$(\frac{31}{63}, \frac{43}{63})$ &
    6 & as above\\ \\

$(\frac{4}{7}, \frac{5}{7})$ &
    3 & as above\\ \\

\end{tabular}
\caption{Proof that $(a_i, b_i)$ is on the boundary of $D_3$}
\label{tab:Corner}
\end{table}

}
To see that $(a_i, b_{i+1})$ is in $D_3$ we must show, for all
    $n \geq 3$, how to construct an $n$-cycle that avoids $(a_i, b_{i+1})$.
These results are summarized in Table \ref{tab:Anti-corner}.
We will consider only one of these cases in detail, all of the rest are
    equivalent.
The second half of Table \ref{tab:Anti-corner} comes by replacing
    all of the $0$s with $1$s and all of the $1$s with $0$s in the
    first half of the table.

Consider the special case of finding a $7$-cycle that avoids
    $(a_1, b_2) = \left(\frac27, \frac{32}{63}\right)$.
We see that $(0100100)^\infty$, a special case of $(010 (010)^* 0)^\infty$,
    is a $7$-cycle.
We see that the $7$ terms in the orbit of $(0100100)^\infty$ are
   \[ (010 010 0)^\infty,
   (10 010 0 0)^\infty,
   (0 010 0 01)^\infty,
   (010 0 010)^\infty,
   (10 0 010 0)^\infty,
   (0 0 010 01)^\infty,
   (0 010 010)^\infty. \]
By looking at the dyadic expansions, we see that the first, third, fourth,
    sixth and seventh term are all strictly less that $\frac27$, whereas the
    second and fifth term are strictly larger than $\frac{32}{63}$.

{\footnotesize

\begin{table}
\begin{tabular}{lll}
$a$                      & $b$                      & $c$ \\ \hline
$0 = (0)^\infty$          & $3/7 = (011)^\infty$      & $(011(1)^*)^\infty$ \\ \\
$2/7 = (010)^\infty$      & $32/63 = (100000)^\infty$ &
       $(010)^\infty$ \\
    && $(010(010)^*0)^\infty $ \\
    && $(010(010)^*00)^\infty $ \\
    && $(010(010)^*000)^\infty $ \\ \\
$20/63 = (010100)^\infty$ & $16/31 = (10000)^\infty$  &
       $(010)^\infty$ \\
    && $(010(010)^*0)^\infty $ \\
    && $(010(010)^*00)^\infty $ \\
    && $(010(010)^*000)^\infty $ \\ \\
$10/31 = (01010)^\infty$  & $8/15 = (1000)^\infty$    &
       $(010)^\infty$ \\
    && $(010(010)^*0)^\infty $ \\
    && $(010(010)^*10)^\infty $ \\
    && $(010(010)^*100)^\infty $ \\ \\
$2/5 = (0110)^\infty$     & $4/7 = (100)^\infty$      &
       $(01(01)^*0)^\infty$ \\
    && $(01(01)^*10)^\infty$ \\ \\
$3/7 = (011)^\infty$      & $3/5 = (1001)^\infty$     &
    By symmetry \\ \\
$7/15 = (0111)^\infty$    & $21/31 = (10101)^\infty$  &
    By symmetry \\ \\
$15/31 = (01111)^\infty$  & $43/63 = (101011)^\infty$ &
    By symmetry \\ \\
$31/63 = (011111)^\infty$ & $5/7 = (101)^\infty$      &
    By symmetry \\ \\
$4/7 = (100)^\infty$      & $1 = 10^\infty$         &
    By symmetry \\ \\
\end{tabular}
\caption{Proof that $(a_i, b_{i+1})$ is an anti-corner}
\label{tab:Anti-corner}
\end{table}
}

This proves the result that $(a_i, b_i)$ for $i = 1, 2, \dots 9$,
    form the corners of $D_3$ and
    $(a_i, b_{i+1})$ for $i = 0, 1, \dots, 9$ form the anti-corners.
\end{proof}

\begin{defn}\label{def:EP}We say that $n\ge3$ is an {\em exit period} if there exists a continuous family of intervals $(a_\al, b_\al)_{\al\in[\al_0,\al_1]}$ such that
\begin{itemize}
\item $(a_\al, b_\al)\subsetneq (a_{\al'}, b_{\al'})$ if $\al>\al'$;
\item $B(a_{\al_1}, b_{\al_1})=\varnothing$;
\item $n$ is bad for $(a_\al,b_\al)$ for any $\al<\al_1$.
\end{itemize}
We denote the set of exit periods by $EP$.
\end{defn}

\begin{prop}\label{prop:exit}
We have $EP=\{3,4,5,6\}$.
\end{prop}
\begin{proof}It is clear from Definition~\ref{def:EP} that $(a_{\al_1}, b_{\al_1})$ must belong to the boundary of $D_3$. It follows from the proof of Theorem~\ref{thm:D3 corner} that for any $\e>0$ there exists $(a,b)$ at a distance $\e$ from $\partial D_3$ such that $B(a,b)\subset\{3,4,5,6\}$. Hence $EP\subset\{3,4,5,6\}$.

To prove that $\{3,4,5,6\} \subset EP$, it suffices to show that for any $n\in\{3,4,5,6\}$ the corresponding corner is indeed $(a_{\al_1}, b_{\al_1})$ for some family of intervals satisfying Definition~\ref{def:EP}. This is a simple check; for instance, for the interval $\left(\frac{20}{63},\frac{32}{63}\right)$ we have that any 6-cycle which does not include its endpoints contains two consecutive 1s in its dyadic expansions and thus, must intersect it. Hence $6\in EP$. The cases of $n=3,4,5$ are similar.
\end{proof}

\begin{rmk}\label{rmk:D3}
The only symmetric point on the boundary of $D_3$, $\left(\frac37,\frac47\right)$, corresponds to the appearance of period~3 in the classical Sharkovski\u{\i} order -- see Introduction. We see thus that $\partial D_3$ can be perceived as a generalization of period~3 to our asymmetric case.
\end{rmk}

\section{The set $D_2$}
\label{sec:D2}

Similar to the boundary of set $D_3$, the boundary of the set $D_2$ consists of horizontal and vertical line segments.
Unlike $D_3$ though, the boundary of this set consists of an infinite number
    of such segments.
Definitions~\ref{def:corner} of corners remains.
Definition~\ref{def:anti-corner} of anti-corners is not relevant in
    this case (although this is not immediately obvious).
If we consider a horizontal line in $D_3$, we see that
    the right end of this line is a corner and
    the left end of this line is an anti-corner.
In the case of $D_2$, the right end of this line is again a corner and the
    left end of this line is a limit of corners, and comes from a kind of
    Devil's staircase construction.
We will make this rigorous in Proposition~\ref{prop:Devil}.

Before discussing the result in detail, we must first introduce some
    additional notation from the combinatorics on words.
We say that a finite  word $u$ is a {\em factor of} $w$ if there exists $k$ such that $u=w_k\dots w_{k+n}$ for some $n\ge0$. For a finite word $w$ let $|w|$ stand for its length and $|w|_1$ stand for the number of 1s in $w$. The 1-{\em ratio} of $w$ is defined as $|w|_1/|w|$. For an infinite word $w_1w_2\dots$ the 1-ratio is defined as $\lim_{n\to\infty}|w_1\dots w_n|_1/n$ (if it exists).

We say that a finite or infinite word $w$ is {\em balanced} if for any $n\ge1$ and any two factors $u,v$ of $w$ of length~$n$ we have $||u|_1-|v|_1|\le1$. 
A finite word $w$ is {\em cyclically balanced} if all of its cyclic permutations are balanced. (And therefore, $w^\infty$ is balanced.) It is well known that if $u$ and $v$ are two cyclically balanced words with $|u|=|v|=q$ and $|u|_1=|v|_1=p$ and $\gcd(p,q)=1$, then $u$ is a cyclic permutation of $v$. Thus, there are only $q$ distinct cyclically balanced words of length $q$ with $p$ 1s.

A finite word $w$ which begins with 0 is called 0-{\em max} if it is larger than any of its cyclic permutations beginning with 0. A finite word is called 1-{\em min} if it smaller than any if its cyclic permutations beginning with 1. Similarly, an infinite word $w=w_1w_2\dots$ with $w_1=0$ is 0-{\em max} if $(w_{k+1}, w_{k+2},\dots)\prec w$ for any $k\ge1$ such that $w_{k+1}=0$. An infinite word $w=w_1w_2\dots$ with $w_1=1$ is 1-{\em min} if $(w_{k+1}, w_{k+2},\dots)\succ w$ for any $k\ge1$ such that $w_{k+1}=1$.

For any $r=p/q\in\mathbb Q\cap(0,1)$ we define two words as follows: $s(r)$ is the lexicographically largest cyclically balanced word of length~$q$ with 1-ratio $r$ beginning with 0, and $t(r)$ is the lexicographically smallest cyclically balanced word of length~$q$ with 1-ratio $r$ beginning with 1. In particular, $s(r)$ is 0-max and $t(r)$ is 1-min.


Note that there is an explicit way to construct $s(r)$ and $t(r)$ for any given $r$. Namely, let $r=p/q\le1/2$ have a continued fraction expansion $[d_1+1,\dots,d_n]$ with $d_n\ge2$ and $d_1\ge1$ (in view of $r\le1/2$). We define the sequence of 0-1 words given by $r$ as follows: $u_{-1}=1, u_0=0,
u_{k+1}=u_k^{d_{k+1}}u_{k-1}, \ k\ge0$. The word $u_n$ has length~$q$ and is called the $n$th {\em standard word} given by $r$. Given an irrational $\gamma\in(0,1/2)$ with the continued fraction expansion $\ga=[d_1+1,d_2,\dots]$, the word $u_\infty$ defined as the limit of the $u_n$ is called the {\em characteristic word} given by $\ga$.

Let $w_1\dots w_q:=u_n$. Then
\begin{equation}\label{eq:srtr}
s(r)=01w_1\dots w_{q-2},\ t(r)=10w_1\dots w_{q-2}.
\end{equation}
For $r\in\mathbb Q\cap(1/2,1)$ we have $s(r)=\overline{t(1-r)}, \ t(r)=\overline{s(1-r)}$, where $\overline 0=1, \overline 1=0$, and $\overline{w_1w_2}=\overline{w_1}\,\overline{w_2}$ for any two words $w_1, w_2$.

\begin{example}We have $s(2/5)=01010,\ t(2/5)=10010, \ s(3/5)=01101,\ t(3/5)=10101$.
\end{example}

\begin{lemma}\label{lem:farey}
Assume $r_1=p_1/q_1$ and $r_2=p_2/q_2$ to be Farey neighbours with $r_1<r_2$ and $r_2 \leq 1/2$, i.e., with $p_2q_1-p_1q_2=1$. Put
\[
r_3:=\frac{p_1}{q_1}\oplus \frac{p_2}{q_2} = \frac{p_1+p_2}{q_1+q_2},
\]
i.e., $r_3$ is the mediant of $r_1$ and $r_2$. Put $s_i=s(r_i), t_i=t(r_i)$ for $i=1,2,3$. Then we have
\begin{equation}\label{eq:concaten1}
s_3=s_2s_1,
\end{equation}
\begin{equation}\label{eq:concaten2}
t_3=t_1t_2,
\end{equation}
\begin{equation}\label{eq:concaten3}
s_3=s_1t_2
\end{equation}
and
\begin{equation}\label{eq:concaten4}
t_3=t_2s_1.
\end{equation}
\end{lemma}
\begin{proof} Let the continued fraction expansion for $r_1$ be $[d_1+1,\dots, d_{m-1}, d_m]$, in which case $m$ is even and $r_2=[d_1+1,\dots, d_{m-1}]$. Then, as is well known, the continued fraction expansion of their mediant $r_3$ is $[d_1+1,\dots, d_{m-1}, d_m+1]$.

Since $m$ is even, the standard words $u_m$ and $u_m'$which correspond to $r_1$ and $r_3$ respectively, end with $10$, while $u_{m-1}$ which corresponds to $r_2$ ends with $01$. Denote $u_m=v_m10, u_{m-1}=v_{m-1}01$ and $u_m'=v_m'10$. We have $u_m = u_{m-1}^{d_m} u_{m-2}$ and $u_m' = u_{m-1}^{d_m+1} u_{m-2} = u_{m-1} u_m$, whence
\[
v_m'10=v_{m-1}01v_m10=v_{m-1}s_110.
\]
Prepending $01$ as a prefix for both sides of this equation and deleting the suffix $10$, we obtain (\ref{eq:concaten1}). Replacing the first two symbols $01$ with $10$ yields (\ref{eq:concaten4}).

To prove (\ref{eq:concaten2}), notice that $(u_{m-1}, u_m)$ is a standard pair (see \cite[Chapter~2.2.1]{Loth}), whence $u_{m-1}u_m$ and $u_mu_{m-1}$ differ only by the last two symbols (\cite[Proposition~2.2.2(iii)]{Loth}). Therefore,
\[
v_{m-1}01v_m=v_m10v_{m-1}.
\]
Prepending $10$ as a prefix for both sides of the equation, we obtain $t_2s_1=t_1t_2$, which in view of (\ref{eq:concaten4}) yields (\ref{eq:concaten2}). Replacing the first two symbols $01$ with $10$ in (\ref{eq:concaten2}) yields (\ref{eq:concaten3}).

The case of $r_1=[d_1+1,\dots, d_{m-1}]$ and $r_2=[d_1+1,\dots, d_{m-1}, d_m]$ implies that $m$ is odd and is treated similarly, so we omit the proof.
\end{proof}

\begin{cor}
\label{cor:limit corners}
We have for $r_n$ and $r$ in $\mathbb Q$ that
\[
\begin{array} {rclrcl}
\lim_{r_n \uparrow r} s(r_n) & = & s(r)^\infty   & \ \ \ \ \ \
\lim_{r_n \uparrow r} t(r_n) & = & t(r) s(r)^\infty   \\
\lim_{r_n \downarrow r} s(r_n) & = & s(r) t(r)^\infty   &
\lim_{r_n \downarrow r} t(r_n) & = & t(r)^\infty
\end{array} \]
\end{cor}
Here $r_n \uparrow r$ is the one-sided limit from the left, and
    $r_n \downarrow r$ the one-sided limit from the right.

\begin{proof}
We observe that the map $r \to s(r_n)$ and $r \to t(r_n)$ are both
    strictly monotonically increasing.
Hence, it is sufficient to show this result on a subsequence of
    $r_n \to r$ from either above or below.
Let $r_0$ be a Farey neighbour of $r$.
Without loss of generality assume $r_0 < r$.
Define recursively $r_n = r \oplus r_{n-1}$.
We see that $s(r_n)  := s(r) s(r_{n-1}) = s(r)^n s(r_0) \to s(r)^\infty$
    and similarly $t(r_n) \to t(r) s(r)^\infty$.
The other cases are similar.
\end{proof}

Given $r=p/q\in(0,1)$, put $s=s(r), t=t(r)$. We know that $s^\infty$ and $t^\infty$ belong to the same $q$-cycle. Furthermore, it follows from \cite[Corollary~3.6]{SSC} that this is the only cycle in $\mathcal J(s^\infty, t^\infty)$.

\begin{rmk}Note that the notation in \cite[Corollary~3.6]{SSC} differs from the present paper. Specifically, $\alpha_r$ in \cite{SSC} stands for $s(r)^\infty$ and $\gamma_r$ for $t(r)^\infty$, where $r=p/q$. The $T$-orbit of $\alpha_r$ is denoted by $\mathcal O(p/q)$.
\end{rmk}

Let $\{s,t\}^{\om}$ denote the set of infinite words which are concatenations of $s$ and $t$ together with its shifts.
That is, \[ \{s,t\}^{\om} = \left\{
    T^k a_1 a_2 a_3 \dots | k \geq 0, a_i \in \{s,t\} \right\} \]

\begin{lemma}\label{lem:qcycles}
We have
\[
\mathcal J(st^\infty, ts^\infty)\cap (s^\infty, t^\infty)\subset\{s,t\}^{\om}\subset \mathcal J(st^\infty, ts^\infty).
\]
\end{lemma}
\begin{proof}Suppose $x\in(s^\infty,t^\infty)\setminus (st^\infty, ts^\infty)$. Then either $x\in(s^\infty, st^\infty]$ or in $[ts^\infty, t^\infty)$. Both cases are similar, so let us assume the former. Then the dyadic expansion of $x$ begins with $s$. If $T^qx\notin(st^\infty, ts^\infty)$, then again, its dyadic expansion begins with either $s$ or $t$, etc. -- see Figure~\ref{fig:Tq}. (Note that $T^q$ acts on the dyadic expansions as the shift by $q$ symbols.) This proves $\mathcal J(st^\infty, ts^\infty)\cap (s^\infty, t^\infty)\subset\{s,t\}^{\om}$.

\begin{figure}[t]
\centering \unitlength=1.0mm
\begin{picture}(40,70)(0,0)

\thinlines

\path(-10,5)(50,5)(50,65)(-10,65)(-10,5)

\dottedline(-10,5)(50,65)
\dottedline(10,5)(10,65)

\put(-12,2){$s^\infty$}
\put(8,2){$st^\infty$}
\put(48,2){$t^\infty$}






\thicklines

\path(-10,5)(10,65)
\path(30,5)(50,65)

\thinlines

\dottedline(30,5)(30,65)

\put(28,2){$ts^\infty$}

\end{picture}

\caption{Part of the map $T^q$. }
    \label{fig:Tq}
  \end{figure}
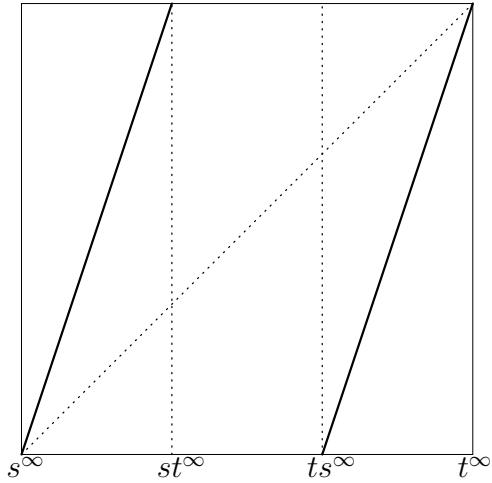

Now let us show that
\begin{equation}\label{eq:converse}
\{s,t\}^{\om}\subset \mathcal J(st^\infty, ts^\infty).
\end{equation}
Let $w\in\{s,t\}^\om$. It suffices to show that for any $j\ge0$ such that $w_{j+1}=0$ we have $(w_{j+1},w_{j+2},\dots)\prec st^\infty$ and for any $j$ such that $w_{j+1}=1$ we have $ts^\infty \prec (w_{j+1},w_{j+2},\dots)$. Both claims are similar, so we will only prove the first.

Let $s=s_1\dots s_q$ and $t=t_1\dots t_q$. We first consider the case $q=2$. Here $r=1/2, s=01$ and $t=10$. It is a simple check that the 0-max of $\{s,t\}^\om$ is $01(10)^\infty=st^\infty$, and the 1-min is $ts^\infty$. This yields (\ref{eq:converse}) for $q=2$.

If $q\ge3$, then $s_k=t_k,\ 3\le k\le q$. Since $s\prec t$, it suffices to show that
\begin{equation}\label{eq:precs}
s_{j+1}\dots s_q (t_1\dots t_q)^\infty\prec st^\infty
\end{equation}
provided $s_{j+1}=0$. Since $s$ is 0-max, we have $s_{j+1}\dots s_q\preceq s_1\dots s_{q-j}$; furthermore, if we have an equality, then $s_{q-j+1}=1$ (\cite[Lemma~5.1]{GS13}). Thus, if $s_{j+1}\dots s_q\prec s_1\dots s_{q-j}$, we are done. Otherwise, (\ref{eq:precs}) will follow from
\begin{equation}\label{eq:prec2}
t_1\dots t_q \prec s_{q-j+1}\dots s_q t_1\dots t_{q-j}.
\end{equation}
Consider first the case $j=q-1$. Here (\ref{eq:prec2}) turns into
\begin{equation}\label{eq:prec3}
t_2\dots t_q\prec t_3\dots t_qt_1.
\end{equation}
Since $t$ is 1-min and $s$ is its cyclic permutation, we have that any of the cyclic permutations of $s$ which begins with 1 is lexicographically larger than $t$.
Hence $t_1\dots t_q\prec s_2\dots s_qs_1$.
By noticing that $t_1 = s_2 = 1$, $t_k = s_k$ for $k = 3, 4, \dots q$ and
    $s_1 = 0 < t_1 = 1$ we get
    $t_2\dots t_q\prec s_3\dots s_qs_1 \prec t_3\dots t_qt_1$ as required.
Hence $t_1\dots t_q\prec s_2\dots s_qs_1$, which implies (\ref{eq:prec3}), as $s_1<t_1$.

Now assume $j\le q-2$. Since $s_k=t_k,\ 3\le k\le q$, this is equivalent to
\[
t_1\dots t_q \prec t_{q-j+1}\dots t_q t_1\dots t_{q-j},
\]
which follows from $t_{q-j+1}=s_{q-j+1}=1$ and $t$ being 1-min.
\end{proof}

\begin{cor}\label{cor:Hdim}
We have
\[
\dim_H \mathcal J(st^\infty, ts^\infty)=\frac1q.
\]
\end{cor}
\begin{proof}It follows from \cite[Corollary~3.6]{SSC} that $\mathcal J(s^\infty, t^\infty)$ is a countable set. When $x\in\mathcal J(st^\infty, ts^\infty)\cap[s^\infty, t^\infty]$, we know that its dyadic expansion is a concatenation of the blocks $s$ and $t$. This means that the topological entropy of $T$ restricted to $\mathcal J(st^\infty, ts^\infty)$ is $\frac1q\log2$, whence the claim follows from the well known formula $$\text{Hausdorff dimension}=\text{topological entropy}/\text{Lyapunov exponent}$$ (see, e.g., the seminal paper \cite{F}) and the fact that the Lyapunov exponent of $T$ is equal to $\log2$.
\end{proof}

\begin{rmk}It is important to state the exact logical dependence between results in \cite{SSC} and the present paper. Namely, Lemma~\ref{lem:farey} $\to$ \cite[Lemma~2.5]{SSC} $\to$
\cite[Corollary~3.6]{SSC} $\to$ Corollary~\ref{cor:Hdim}.
\end{rmk}

Now we are ready to prove the main result of this section.

\begin{thm}\label{thm:D2}
Let $s := s(r), t := t(r)$ for $r \in \mathbb{Q} \cap (0,1)$.
For all $r \in \mathbb{Q} \cap (0,1)$ we have that
    $(s^\infty, t s^\infty), (s t^\infty, t s^\infty)$ and
    $(s t^\infty, t^\infty)$ are on the boundary of $D_2$.
In particular, $(s t^\infty, t s^\infty)$ is a corner point.
Furthermore, $(s^\infty, t s^\infty)$ and $(s t^\infty, t^\infty)$
    are the limit points of corner points.
\end{thm}

\begin{proof}
We will first show that $(s t^\infty, t s^\infty)$ is on the boundary
    of $D_2$.
Then, using Corollary~\ref{cor:limit corners} we get that
    $(s^\infty, t s^\infty)$ and
    $(s t^\infty, t^\infty)$ are also on the boundary of $D_2$
    by noting that
    \[ \lim_{r_n\uparrow r} (s(r_n) t(r_n)^\infty, t(r_n) s(r_n)^\infty) =
    (s(r)^\infty, t(r) s(r)^\infty)\] and
    \[ \lim_{r_n\downarrow r} (s(r_n) t(r_n)^\infty, t(r_n) s(r_n)^\infty) =
    (s(r) t(r)^\infty, t(r)^\infty).\]
This then implies that $(s t^\infty, t s^\infty)$ is a corner point.

Let us prove first that $(st^\infty, ts^\infty)$ is on the boundary of $D_2$. Note first that by Lemma~\ref{lem:qcycles},  $\mathcal J(st^\infty, ts^\infty)$ contains only cycles whose lengths are multiples of $q$, we have that  $(st^\infty, ts^\infty)\notin D_2$. Consequently, $(st^\infty-\e, ts^\infty+\e)\notin D_2$ for any $\e>0$.

To see for arbitrarily small $\e > 0$ that $(st^\infty + \e, t s^\infty)$ and $(st^\infty, t s^\infty -\e)$ are in $D_2$ we must show that there exists an $N$ (dependent on $\e$) such that for all $n > N$ there exists an $n$-cycle that is disjoint from $(st^\infty , t s^\infty- \e)$ and an $n$-cycle that is disjoint from $(st^\infty + \e, t s^\infty)$. We will show the first case only, the second case is symmetric.

Since $s$ and $t$ are cyclic permutations of each other, we can write $s = u v$ and $t = v u$. More precisely, (\ref{eq:concaten3}) and (\ref{eq:concaten4}) yield explicit $u$ and $v$ with $\gcd(|u|,|s|)=\gcd(|v|,|s|)=1$.
We will first show that the orbit of $w := (ut^m)^\infty$ is disjoint from $(st^\infty, t s^\infty - \e)$ for all sufficiently large $m$.

Let $q=|s|=|t|$, where $r=p/q$ and $j=|u|$. Write
\begin{eqnarray}
w & = & (u t^m)^\infty \nonumber \\
  & = & (u (vu)^m)^\infty \nonumber \\
  & = & \underbrace{(uv) (uv) \dots (uv)}_{m} u u \dots \nonumber \\
  & = & \underbrace{s s \dots s}_{m} u u \dots \label{eq:1} \\
  & = & u \underbrace{(vu) (vu) \dots (vu)}_{m}
          \underbrace{(uv) (uv) \dots (uv)}_{m} uu \dots  \nonumber \\
  & = & u \underbrace{t t \dots t}_{m} \underbrace{s s \dots s}_m u u \dots
        \label{eq:2}
\end{eqnarray}
Let as usual, $\text{dist}(x,y)=2^{-\min\{j\ge1\,\mid\, x_j\neq y_j\}}$ for any pair $x,y\in \{0,1\}^{\mathbb N}$.
Letting $x = x_1 x_2 \dots,
        y = y_1 y_2 \dots,  \in \{0,1\}^{\mathbb N}$ we see that
   $|\sum x_i/2^i - \sum y_i/2^i | \leq 2 \text{dist}(x,y)$.

By \eqref{eq:1}, we have $\text{dist}(T^i w, T^is^\infty)\le 2^{-mq}$ for all $i\le j$. Furthermore, since $w= s^m uu \dots$ and $s^\infty = s^m uv \dots$, we have $T^i w\prec T^i s^\infty$ for all $i\le j$. Lemma~\ref{lem:qcycles} yields $s^\infty \in \mathcal{J} (s t^\infty, t s^\infty)$, whence for $m$ sufficiently large the first $j$ terms in the orbit of $w$ are disjoint from $(s t^\infty, t s^\infty - \e)$.

By \eqref{eq:2}, we have $\text{dist}(T^i w, T^{i-j}(t^ms^\infty))\le 2^{-mq}$ and $T^i w\prec  T^{i-j}(t^ms^\infty)$ if $j+1\le i\le mq+j$. Again, by Lemma~\ref{lem:qcycles}, $t^m s^\infty \in \mathcal{J} (s t^\infty, t s^\infty)$, whence for $m$ sufficiently large the $j+1$-st term to the $mq$-th in the orbit of $w$ are disjoint from $(s t^\infty, t s^\infty - \e)$.

Thus, we have proved that for $m$ sufficiently large the orbit of $w= (ut^m)^\infty$ (whose length is $mq+j$) is disjoint from $(st^\infty, t s^\infty - \e)$. Now we will show for all $\ell\in\{0,1,\dots,q-1\}$ there exists a word $W$ of length $mq+\ell$ (for all $m$ sufficiently large) such that the orbit of $W$ is disjoint from $(st^\infty, t s^\infty - \e)$.

Since $j$ is coprime with $q$, for all $\ell$ there exists a $k$ such that $\ell \equiv k j \bmod q$. We now let $m_1, m_2, \dots m_k$ be sufficiently large, and distinct,
    such that each $(u (vu)^{m_i})^\infty$ is disjoint from $(st^\infty, ts^\infty - \e)$. Consider
    \[
    W := (u (vu)^{m_1} u (vu)^{m_2} \dots u (vu)^{m_k})^\infty.
    \]
The same argument as before shows that the orbit of $W$ is disjoint from $(st^\infty, t s^\infty - \e)$. Furthermore, $|W|\equiv \ell\bmod q$, which concludes the proof of the first part of the claim.

This shows that all neighbourhoods of $(s t^\infty, t s^\infty)$ have points
    in $D_2$ and points not in $D_2$, and hence it is a boundary point.

To see it is a corner point, we consider $r_n \uparrow r$ and notice that
    $(s(r_n) t(r_n)^\infty, t(r_n) s(r_n)) \to (s(r)^\infty, t(r) s(r)^\infty)$
    must also be on the boundary.
Similarly $(s t^\infty, t^\infty)$ is a boundary point.
This implies that $(s t^\infty, t s^\infty)$ is a corner point.
\end{proof}

For a sketch of $D_2$ see Figure~\ref{fig:D2_labelled}.
For the purposes of that diagram,
    let $p(r) = (s(r) t(r)^\infty, t(r) s(r)^\infty)$ and
    $p'(r) = (s(r)^\infty, t(r) s(r)^\infty)$.
We notice that visually $p(n/(2n+1)) \to p'(1/2)$, as proven theoretically
    in Theorem~\ref{thm:D2}.

\begin{cor}If $b-a<\frac16$, then $(a,b)\in D_2$, and the constant $\frac16$ cannot be improved.
If $(a,b) \in D_2$ then $b- a < \frac{1}{4}$ and the constant $\frac{1}{4}$
    cannot be improved.
\end{cor}
\begin{proof}
As usual, let $t = t(r)$ and $s = s(r)$ for some $r$.
We have
\begin{align*}
\inf \{b-a : (a,b)\in D_2\} &= \min \{b-a : (a,b)\ \text{is a corner of}\ D_2\} \\
&=ts^\infty - st^\infty = ts^\infty - s^\infty - (st^\infty - s^\infty)\\
&= t - s -2^{-q}(t^\infty - s^\infty) = \frac14 - 2^{-q}(t^\infty - s^\infty)\\
&= \frac14 - \frac1{4(2^q-1)} = \frac{2^q-2}{4(2^q-1)},
\end{align*}
and its minimum is attained at $q=2$ and is equal to $\frac16$. Clearly, $\frac16$ cannot be improved, in view of $\left(\frac5{12},\frac7{12}\right)$ being a corner of $D_2$.

We similarly have
\begin{align*}
\sup \{b-a : (a,b)\in D_2\}
&=\max\{b-a: (a,b)\ \text{is the left endpoint of a horizontal line}\}\\
&=ts^\infty - s^\infty  = t - s  \\
&= \frac14.
\end{align*}
\end{proof}
It is interesting to note that the minimum occurs in only one location
    whereas the maximum occurs in infinitely many places.

\begin{rmk}
Combined with the results of \cite{GS13}, we now have four sharp constants $c_0=\frac14,\ c_1=\frac14\prod_{n=1}^\infty \bigl(1-2^{-2^n}\bigr)\approx 0.175092,\ c_2=\frac16$ and $c_3=\frac2{15}$ such that if $b-a<c_j$, then $(a,b)\in D_j$ for $j=0,1,2,3$.
\end{rmk}

Put for $\frac14\le a\le \frac12$,
\[
\varkappa(a)=\sup\{b: \mathcal J(a,b)\ \text{contains infinitely many cycles}\}.
\]

\begin{prop}\label{prop:Devil}
We have
\[
\left\{(a,b)\in \partial D_2 : \frac14 < a \le \frac1{2}\right\}
    = \mathrm{cl}\left(
      \left\{(a,\varkappa(a)) : \frac14 < a \le \frac1{2}\right\}
     \cup \left\{(\varkappa(a), a) : \frac14 < a \le \frac1{2}\right\}\right),
\]
 where $\varkappa(a)$
\begin{enumerate}
\item is non-decreasing;
\item is constant almost everywhere;
\item has jump discontinuities at $a = s(r) t(r)^\infty$ for every
    $r \in \mathbb{Q} \cap (0,1)$.
\end{enumerate}
\end{prop}

\begin{rmk}This is a kind of Devil's staircase.
The traditional Devil's staircase is continuous non-decreasing function
   that is constant almost everywhere.
Here the first restriction of continuity is relaxed,
    allowing instead jump discontinuities at $t(r) s(r)^\infty$.
    \end{rmk}

\begin{proof}
Item~(i) is obvious from the definition of $\varkappa$; item~(iii) follows from the fact that $(s(r)t(r)^\infty, t(r)s(r)^\infty)$ and
$(s(r)t(r)^\infty, t(r)^\infty)$ are both on the boundary of $D_2$ (see Theorem~\ref{thm:D2}), so we have a jump. Let us prove (ii).

By the above, we may define $\varkappa$ at $s(r) t(r)^\infty$ either as $t(r) s(r)^\infty$ or $t(r)^\infty$ as
    both will give the same result in the closure.
Following \cite{SSC}, we introduce the set $\mathcal S$ which is defined as a union of the points in $(1/4, 5/12)$ whose dyadic expansion is of the form $01w$, where $w$ is the characteristic word for some irrational $\gamma\in(0,1/2)$ (an uncountable set) and a countable set $\bigcup_{r\in(0,1/2)\cap\mathbb Q} \{s(r)^\infty, s(r)t(r)^\infty\}$. This set is related to the exceptional set $\mathcal E$ for our 
staircase in the following way (see \cite[Section~3]{SSC}):
\[
\left(\frac14, \frac5{12}\right)\setminus \bigcup_{r\in(0,1/2)\cap\mathbb Q} (s(r)^\infty, s(r)t(r)^\infty) = \mathcal S\cap \left(\frac14, \frac5{12}\right).
\]
Consequently,
\begin{align*}
\mathcal E:&=\left(\frac14, \frac5{12}\right)\setminus \bigcup_{r\in(0,1/2)\cap\mathbb Q} [s(r)^\infty, s(r)t(r)^\infty] \\
&= \Bigl(\mathcal S\setminus \bigcup_{r\in(0,1/2)\cap\mathbb Q} \{s(r)^\infty, s(r)t(r)^\infty\}\Bigr)\cap \left(\frac14, \frac5{12}\right).
\end{align*}
The result now follows from the fact that $\mathcal S$ has zero measure. (In fact, even zero Hausdorff dimension -- see \cite[Corollary~3.12]{SSC}.)
\end{proof}

\begin{figure}
\includegraphics[width=250pt,height=313pt]{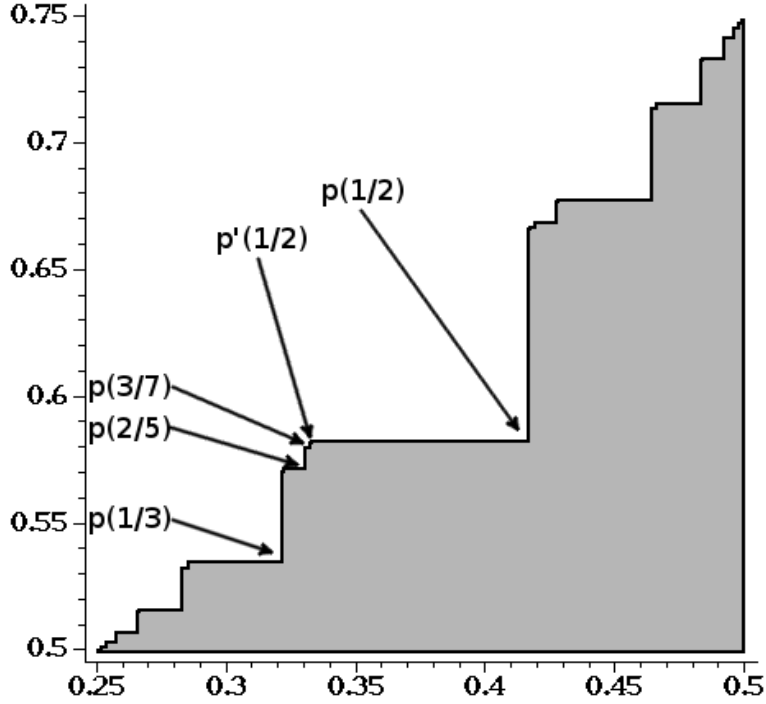}
\caption{The set $D_2$}
\label{fig:D2_labelled}
\end{figure}

Since the length of the plateau given by $r=p/q$ is of order $\frac14\cdot 2^{-q}$, one can say colloquially that the exceptional set corresponds to the case $q\to+\infty$.

\begin{cor}The set $D_2$ is open. Consequently, $\partial D_2\cap\partial D_3=\varnothing$.
\end{cor}
\begin{proof}Fix $r$ and denote $s=s(r), t=t(r)$. As mentioned in the proof of Theorem~\ref{thm:D2}, $(st^\infty, ts^\infty)\notin D_2$, whence the same is true for the plateaus, i.e., $([s^\infty, st^\infty]\times \{ts^\infty\})\cap D_2=\varnothing$ and $(\{st^\infty\}\times [ts^\infty,t^\infty])\cap D_2=\varnothing$ otherwise. As for the exceptional set, it is known that for any given irrational $\gamma\in (0,1/2)$ its characteristic word is aperiodic (see, e.g., \cite[Chapter~2]{Loth}), whence this set cannot contain any cycles.

The second claim is a direct consequence of $D_2$ being open, $D_3$ being closed and $D_3\subset D_2$.
\end{proof}

\begin{rmk}Figure~\ref{fig:D0123} suggests that the two closest points on the boundaries of $D_2$ and $D_3$ are the anti-corner  $(10/31, 8/15)$ of $D_3$ and the corner $(9/28, 15/28)$ of $D_2$, whence the distance between $\partial D_2$ and $\partial D_3$ is equal to $\frac{\sqrt{1186}}{13020}\approx 0.002645$. We leave this a conjecture for the interested reader.
\end{rmk}

\begin{rmk}\label{rmk1}
It is interesting to compare the boundaries of $D_0, D_1$ and $D_2$. It follows from \cite[Proposition~2.6]{GS13} that $\partial D_0\cap(1/4,5/12)$ is also the graph of a function of $a$, denoted by $\phi$, which has exactly the same plateau regions as $\varkappa$. However, $\phi(a)\equiv t(r)^\infty$ on $[s(r)^\infty, s(r)t(r)^\infty]$, whereas, as we know, $\varkappa(a)\equiv t(r)s(r)^\infty$ on the same segment, which is strictly less. As $q$ grows, these values tend to the same limit, whence $\phi$ and $\varkappa$ coincide on $\mathcal E$.

For $D_1$ the corresponding function, $\chi$, is also a kind of Devil's staircase on $(1/4, 5/12)$, however, it has a significantly more complicated set of plateau regions. Nonetheless, it follows from \cite[Theorem~2.13]{GS13} that $\chi(a)\equiv \varkappa(a)$ on $[s(r)^\infty, s(r)t(r)s(r)^\infty]$. In particular,
\[
\partial D_0\cap \partial D_1 \cap \partial D_2= \left\{\Bigl(a,a+\frac14\Bigr) : a\in\mathcal E\right\}.
\]
See Figure~\ref{fig:D0123} for hints of more details.
\end{rmk}

\begin{figure}
\includegraphics[width=250pt,height=313pt]{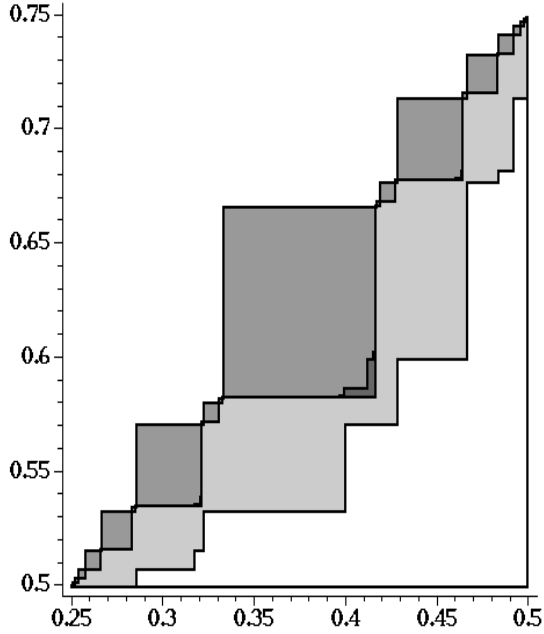}
\caption{The sets
     $D_3$ (lower white region),
     $D_2$ (white + light grey),
     $D_1$ (white + light grey + dark grey), and
     $D_0$ (white + light grey + dark grey + grey)
}
\label{fig:D0123}
\end{figure}

Finally, we would like to describe all possibilities for a ``final stretch'', i.e., all possible sequences in $B(a,b)$ when we descend from $\partial D_2$ towards $\partial D_3$. By definition, when $(a,b)\in D_2$, we have that $B(a,b)$ is finite, which means that $\mathcal J(a,b)$ gains all cycles of all lengths, except possibly a finite number of them, immediately after moving away from the boundary of $D_2$. Thus, if we ignore this finite number of cycles (which we will), it suffices to study $B(a,b)$ for $(a,b)\in\partial D_2$ to determine which cycle lengths we already have on $\partial D_2$.

Put $\BN_L=\{L, L+1, L+2,\dots\}$ for any $L\ge 3$. Note first that if $(a,b)$ is in the exceptional set $\mathcal E$, then $B(a,b)=\BN_3$, since $\mathcal J(a,b)$ does not contain any cycles.

Assume first that $(a,b)$ is on a horizontal plateau $[s^\infty,st^\infty]\times\{ts^\infty\}$ for some $r=p/q$ and $s=s(r), t=t(r)$. If $a=s^\infty$, then, as we know, $\mathcal J(a,b)$ contains only a $q$-cycle; in fact, the same is true for all $a\in [s^\infty, sts^\infty]$, since for any of those there exists $k\ge1$ such that $T^{kq}(a)\in[st^\infty, ts^\infty]$ (see Figure~\ref{fig:Tq}), which implies $\mathcal J(a,b)=\mathcal J(s^\infty, t^\infty)$.

Now assume $a\in(sts^\infty, st^\infty]$. There exists $N$ such that $a>(sts^{n-2})^\infty$ for all $n\ge N$. We claim that if $a>(sts^{n-2})^\infty$, then the orbit of $(sts^{n-2})^\infty$ is contained in $\mathcal J((sts^{n-2})^\infty, ts^\infty)$, which follows from $(sts^{n-2})^\infty$ being a 0-max -- a claim which is proved in a way similar to the proof of (\ref{eq:converse}) (using (\ref{eq:precs})), so we leave it to the interested reader.

This implies that $\mathcal J(a,b)$ contains cycles of all sufficiently large lengths which are multiples of $q$. In view of Lemma~\ref{lem:qcycles}, $\mathcal J(a,b)$ does not contain any cycle of length $qn+j$ for $j\neq0$. The case of vertical plateaus is analogous, so we omit it.

Thus, we have two essentially different possibilities for a ``final stretch'':

\begin{enumerate}
\item $B(a,b)\cap \BN_L=\BN_L$ for some $L\ge3$;
\item $B(a,b)\cap \BN_L = (\BN_L\setminus q\BN)$ for some $L\ge3$ and some $q\ge2$.
\end{enumerate}

The classical Sharkovski\u{\i} order corresponds to the second case with $q=2$.

\nocite{GlendinningSidorov14}

\end{document}